\documentclass{article}
\usepackage{amsfonts,amssymb,amsmath,amsthm}
\usepackage{url}
\usepackage{enumerate}
\usepackage{textcomp}
\usepackage[utf8]{inputenc}

\usepackage{mathtools}
\usepackage{fancyhdr}
\usepackage[all]{xy}

\newcommand{\abso}[1]{\left|#1\right|}
\newtheorem{theorem}{Theorem}
\newtheorem{definition}[theorem]{Definition}
\newtheorem{corollary}[theorem]{Corollary}
\newtheorem{proposition}[theorem]{Proposition}
\newtheorem{lemma}[theorem]{Lemma}
\newtheorem{remark}[theorem]{Remark}

\newcommand{\e}{\varepsilon}
\newcommand{\bine}{\mathbin{\varepsilon}}

\newcommand{\cB}{\mathcal{B}}
\newcommand{\cC}{\mathcal{C}}
\newcommand{\cD}{\mathcal{D}}
\newcommand{\cE}{\mathcal{E}}
\newcommand{\cH}{\mathcal{H}}
\newcommand{\cK}{\mathcal{K}}
\newcommand{\cL}{\mathcal{L}}
\newcommand{\cM}{\mathfrak{M}}
\newcommand{\cO}{\mathcal{O}}
\newcommand{\cS}{\mathcal{S}}

\let\pd\partial

\newcommand{\bN}{\mathbb{N}}
\newcommand{\bR}{\mathbb{R}}

\DeclareMathOperator{\id}{id}
\DeclareMathOperator{\sign}{sign}
\newcommand{\coleq}{\coloneqq}

\date{\today}

\begin{document}

\title{Convolvability and regularization of distributions}

\author{C. Bargetz\footnote{Institut für Mathematik, Universität Innsbruck, Technikerstraße 13, 6020 Innsbruck, Austria, e-mail: christian.bargetz@uibk.ac.at}, E.~A. Nigsch\footnote{Institut für Mathematik, Universität Wien, Oskar-Morgenstern-Platz 1, 1090 Wien, Austria, e-mail: eduard.nigsch@univie.ac.at}, N. Ortner\footnote{Institut für Mathematik, Universität Innsbruck, Technikerstraße 13, 6020 Innsbruck, Austria}}

\maketitle

\begin{abstract}
We apply L.~Schwartz' theory of vector valued distributions in order to simplify, unify and generalize statements about convolvability of distributions, their regularization properties and topological properties of sets of distributions. The proofs rely on propositions on the multiplication of vector-valued distributions and on the characterization of the spaces $\mathcal{O}_{M}(E,F)$ and $\mathcal{O}_{C}'(E,F)$ of multipliers and convolutors for distribution spaces $E$ and $F$.

\end{abstract}

%
%
%



\section{Introduction}

This work has three objectives:

\begin{enumerate}[(a)]
 \item \label{aim1} The generalization of statements about the equivalence of various definitions of \emph{convolvability of distributions}.
 \item \label{aim2} The unification and simplification of \emph{regularization properties}.
 \item \label{aim3} The simplification of proofs of characterizations of \emph{topological properties} (boundedness and relative compactness) of sets of distributions and kernels.
\end{enumerate}

The simplifications alluded to in (\ref{aim2}) and (\ref{aim3}) consist in the use of the theory of vector valued distributions as presented in \cite{mixed,zbMATH03145499,zbMATH03145498}.
This way, one can avoid the use of parametrices on which previous proofs rest (cf., e.g., \cite[p.~202]{TD} and \cite[p.~544]{zbMATH03163214}). The reformulation of classical statements referred to in (\ref{aim1}) and (\ref{aim2}) is done by means of de Wilde's closed graph theorem.

Let us now describe these aims and their background in more detail.

\textbf{ad (\ref{aim1}): convolvability.} Textbooks on distribution theory mostly mention only the following convolution mappings (for distribution spaces on $\bR^n$):
\begin{itemize}
 \item $\cE' \times \cD' \xrightarrow{*} \cD'$, well-defined by condition~($\Sigma$), see~\cite[p.~383]{zbMATH03230708},
 \item $\cD'_{+\Gamma} \times \cD'_{+\Gamma} \xrightarrow{*} \cD'_{+\Gamma}$, well-defined by a cone-condition (C),
 \item $\cO_C' \times \cS' \xrightarrow{*} \cS'$, well defined by a growth condition (D).
\end{itemize}
While $(\Sigma)$ and (C) impose conditions on the supports, (D) restricts the growth of the distributions to be convolved. Already very basic examples in $\bR^2$ show that these mappings are not sufficient to treat certain problems by convolution. Indeed, choose for $\Gamma$ the forward light cone 
$\Gamma = \{\, (t,x) \in \bR^2\ |\ t \ge \abso{x}\,\}$
and consider the fundamental solutions $Y(t) \delta(t+x)$ of $\pd_t - \pd_x$ and $Y(t)\delta(t-x)$ of $\pd_t + \pd_x$ (where $Y$ is the Heaviside function). Convolving them by (C) then gives the uniquely determined fundamental solution $\frac{1}{2} Y(t - \abso{x})$ of $\pd_t^2 - \pd_x^2$ with support in $\Gamma$. However, if we choose instead the two-sided fundamental solutions $\frac{1}{2} \sign(t) \delta(t+x) = (Y(t) - 1/2) \delta(t+x)$ of $\pd_t - \pd_x$ and
 $\frac{1}{2} \sign(t)\delta(t-x) = (Y(t) - 1/2) \delta(t-x)$ of $\pd_t + \pd_x$,
which both have support in $\Gamma \cup (-\Gamma)$, these are convolvable by neither of the conditions ($\Sigma$), (C) and (D).

For this reason, more general definitions of convolution and convolvability were given already 1954 in \cite[exp.~22]{semschwartz} and 1974 in \cite{zbMATH03469533}: two distributions $S,T \in \cD'(\bR^n)$ are called \emph{convolvable} if
\begin{enumerate}
 \item[(SH)] $\forall \varphi \in \cD$: $\varphi(x+y) S(x)T(y) \in \cD'_{L^1}(\bR^{2n}_{xy})$, or
 \item[(S)] $\forall \varphi \in \cD$: $(\varphi * \check S)T \in \cD'_{L^1}(\bR^n)$.
\end{enumerate}
The equivalence of (S) to partial summability in $y$ of the kernel $S(x-y)T(y)$, i.e., $S(x-y) T(y) \in \cD'_x \widehat\otimes \cD'_{L^1,y}$, was established in \cite[p.~132]{zbMATH03145498}. The implication (SH) $\Rightarrow$ (S) was shown in \cite{semschwartz} using the theorem of Fubini for kernels in $\cD'_{L^1,xy}$.
The converse (S) $\Rightarrow$ (SH) was shown in 1959 by R.~Shiraishi (\cite{zbMATH03149320}) and independently in 1976 by B.~Roider (\cite{zbMATH03527321}).

Following Y.~Hirata and R.~Shiraishi (\cite{zbMATH03163214}) one can also formulate condition (SH) without test functions as in
\[ \delta(z-x-y) S(x)T(y) \in \cD'_z \widehat\otimes \cD'_{L^1,xy} = \cL_b(\cD_z, \cD'_{L^1,xy}), \]
where the index $b$ refers to the topology of uniform convergence on bounded subsets of $\mathcal{D}_z$.
Because $\cD'_{L^1, xy}$ is invariant under linear nondegenerate coordinate transformations
the equivalence (S) $\Leftrightarrow$ (SH) can then be written as
\[ S(x-y)T(y) \in \cD'_x \widehat\otimes \cD'_{L^1, y} \Longleftrightarrow \delta(z-x) S(x-y) T(y) \in \cD'_z \widehat\otimes \cD'_{L^1,xy}. \]
Replacing the special kernels $S(x-y)T(y)$ by kernels $K(x,y) \in \cD'_{xy}$, we obtain the general formulation
\[ K(x,y) \in \cD'_x \widehat \otimes \cD'_{L^1,y} \Longleftrightarrow \delta(z-x) K(x,y) \in \cD'_z \widehat\otimes \cD'_{L^1,xy}. \]
Section \ref{sec_kernels} is dedicated to the generalization of statements of this kind.

\textbf{ad (\ref{aim2}): regularization.} A classical regularization statement for $S \in \cD'(\bR^n)$ is the equivalence $S \in \cD'_{L^1} \Leftrightarrow \forall \varphi \in \cD: \varphi * S \in \cD_{L^1}$, which by means of the closed graph theorem can be written in a test function free way as
\[ S \in \cD'_{L^1} \Leftrightarrow S(x-y) \in \cD'_y \widehat\otimes \cD_{L^1,x}. \]
A generalization of this statement to $E$-valued kernels $K(x,y) \in \cD'_y ( E_x)$ is
\[ K(x,y) \in \cD'_{L^1,y}(E_x) \Longleftrightarrow K(x, y-z) \in \cD'_z \widehat\otimes \cD_{L^1,y} (E_x).
\]
Applying this equivalence to $K(x,y) = S(x-y) T(y)$ for $S,T \in \cD'(\bR^n)$ by setting $E = \cD'$,
one obtains
\begin{align*}
S(x-y)T(y) \in \cD'_x \widehat\otimes \cD_{L^1,y} & \Longleftrightarrow S(x-y+z)T(y-z) \in \cD'_z \widehat\otimes \cD'_x \widehat\otimes \cD_{L^1,y} \\
& \hphantom { \Longleftrightarrow S(x-y+z)T(y-z)\,} = \cD'_{xz} \widehat\otimes \cD_{L^1,y} \\
& \Longleftrightarrow S(\xi - y)T(y-\eta) \in \cD'_{\xi\eta} \widehat\otimes \cD_{L^1,y}
\end{align*}
where we applied the kernel theorem, i.e., $\mathcal{D}_{z}'\widehat{\otimes}\mathcal{D}'_{x} = \mathcal{D}'_{zx}$, in the first line and the linear change of coordinates $x + z = \xi, z = \eta$ for the second equivalence. This proves (again by the closed graph theorem) the equivalence of (S) to Chevalley's following convolvability condition given in 1951:
\[ \forall \varphi, \psi \in \cD: (\varphi * \check S)(\psi * T) \in L^1. \]
This equivalence was first shown in 1959 by R.~Shiraishi (\cite{zbMATH03149320}). Section \ref{sec_reg} is concerned with generalizations of such regularization properties.

Our notation and terminology follows \cite{mixed,zbMATH03145499,zbMATH03145498,TD}. However, instead of $K(\hat x, \hat y)$ (as in \cite{zbMATH03145499,zbMATH03145498,zbMATH03163214}) we write $K(x,y)$ for kernels $K \in \cD'(\bR^{2n}_{xy})$. 

The symbol ``$\hookrightarrow$'' denotes continuous injection.

An early version of this article is based on a talk given by the third author at the conference on Generalized Functions, Vienna, September 2009.

\section{Generalization of the equivalence of definitions of convolvability}\label{sec_kernels}

\begin{proposition}\label{prop_eins}
Let $\cH$ and $\cL$ be spaces of distributions such that $\cH$ has topology $\gamma$, $\cH$ and $\cH'_c=\cH_b'$ are nuclear, normal spaces of distributions, $\cH_z \widehat\otimes \cH_x \cong \cH_{zx}$, $\cH$ is invariant under linear nondegenerate coordinate transformations, and there exists a hypocontinuous multiplication mapping $\cH \times \cH_c' \to \cL$. (Note that the subscript $c$ refers to the topology of uniform convergence on absolutely convex compact subsets of $\cH$).

Moreover, suppose that $E$ is a normal space of distributions such that
\begin{equation}\label{propeins_cond} 
\cL_x \bine E_y \hookrightarrow E_{xy} \hookrightarrow \cH_x \bine E_y.
\end{equation}
Then for $K \in \cD'_{xy}$ we have:
\[ K(x,y) \in \cH_x \widehat\otimes E_y \Leftrightarrow \delta(z-x) K(x,y) \in \cH_z \widehat\otimes E_{xy}. \]
\end{proposition}

\begin{remark}
  Note that the assumption $\cH_z \widehat\otimes \cH_x \cong \cH_{zx}$ implies the completeness of $\cH$ and that the space $\cH$ has the topology $\gamma$ iff $(\cH'_c)'_c=\cH$. For a detailed description of this topology, we refer to~\cite[p.~17]{zbMATH03145498}
\end{remark}

\begin{proof}
 ``$\Rightarrow$'': Using our assumptions, by \cite[Prop.~25, p.~120]{zbMATH03145499} there exists a multiplication mapping $\cH'_{c,x} (\cH_z) \times \cH_x (E_y) \to \cL_x ( E_y \widehat\otimes_\pi \cH_z)$. Because $\id \in \cL_\e(\cH, \cH)$ has kernel $\delta(z-x) \in \cH'_{c,x}(\cH_z)$ we obtain
 \[ K(x,y) \cdot \delta(z-x) \in \cL_x \bine (E_y \widehat\otimes \cH_z) = \cL_x \bine E_y \bine \cH_z \hookrightarrow E_{xy} \bine \cH_z = E_{xy} \widehat\otimes \cH_z. \]
 
 Conversely, we have
 \[ \delta(z-x) K(x,y) \in \cH_z \bine E_{xy} \hookrightarrow \cH_z \bine \cH_x \bine E_y = (\cH_z \widehat\otimes \cH_x) \bine E_y = \cH_{zx} \bine E_y. \]
 After applying the linear coordinate transformation $z-x = \xi$, $x=\eta$, we obtain $\delta(\xi)K(\eta,y) \in \cH_{\xi\eta} \widehat\otimes E_y \cong \cH_\xi \widehat\otimes \cH_\eta \widehat\otimes E_y$ 
 which results in $K(x,y) \in \cH_x \widehat\otimes E_y$.
\end{proof}

\begin{remark}
The last part of the proof shows that in case the equivalence holds, $\delta$ necessarily is an element of $\cH$.
\end{remark}

The first part of requirement \eqref{propeins_cond} is obtained in many cases from the following Lemma.
\begin{lemma}\label{lemma_eins}
 Let $\cL \overset{\iota_1}{\hookrightarrow} E \overset{\iota_2}{\hookrightarrow} \cH$, $\cL$ nuclear, $E$ complete, $E_x\times E_y\overset{\otimes}{\to} E_{xy}$ continuous, and assume that $\cH_z \widehat\otimes_\e \cH_x \cong \cH_{zx}$. Then $\cL_x \widehat\otimes E_y \hookrightarrow E_{xy}$.
\end{lemma}
\begin{proof}
 Consider the following diagram.
 \[
  \xymatrix{
  \cL_x \widehat\otimes E_y \ar[r]^{\iota_1 \widehat\otimes_\pi \id_E} \ar[d]_{(\iota_2 \circ \iota_1)\widehat\otimes_\e \iota_2} & E_x \widehat\otimes_\pi E_y \ar[r]^{\widehat \otimes } & E_{xy} \ar[d]^{\iota_2} \\
  \cH_x \widehat\otimes_\e \cH_y \ar[rr]_{{\rm id}} & & \cH_{xy}
  }
 \]
The map in question is $\widehat \otimes \circ (\iota_1 \widehat\otimes_\pi \id_E)$, which is injective if and only if its composition with $\iota_2$. is. By commutativity of the diagram, this is equivalent to injectivity of ${\rm id} \circ ((\iota_2 \circ \iota_1) \widehat\otimes_\e \iota_2)$, which we have by assumption.
\end{proof}

Stated with (1) $\cL = \cE'$, $\cH = \cD'$ or (2) $\cL = \cO_C'$, $\cH = \cS'$, Proposition \ref{prop_eins} reads as follows:
\begin{corollary}
 Suppose $E$ is a complete, normal space of distributions such that $E_x \times E_y \overset{\otimes}{\to} E_{xy}$ is continuous, and $K \in \cD'_{xy}$.
 \begin{enumerate}[(1)]
  \item If $\cE' \hookrightarrow E$, $E_{xy} \hookrightarrow \cD'_x \widehat\otimes E_y$, the following equivalence holds:
  \[ K(x,y) \in \cD'_x \widehat\otimes E_y \Longleftrightarrow \delta(z-x) K(x,y) \in \cD'_z \widehat\otimes E_{xy}. \]
  \item If $\cO_C' \hookrightarrow E \hookrightarrow \cS'$, $E_{xy} \hookrightarrow \cS_x' \widehat\otimes E_y$, the following equivalence holds:
  \[ K(x,y) \in \cS'_x \widehat\otimes E_y \Longleftrightarrow \delta(z-x) K(x,y) \in \cS'_z \widehat\otimes E_{xy}. \]
 \end{enumerate}
\end{corollary}

\begin{remark}
 We give some remarks concerning the assumptions of Proposition \ref{prop_eins}.
\begin{enumerate}[(a)]
 \item For $E = \dot\cB$, $\cC_0$ and $\dot\cB'$ the assumption $E_{xy} \hookrightarrow \cS'_x \widehat\otimes E_y$ is satisfied because $\dot\cB_{xy} = \dot\cB_x \widehat\otimes_\e \dot\cB_y$ (\cite[Prop.~28, p.~98 and Prop.~17, p.59]{zbMATH03145498}), $\cC_{0,xy} = \cC_{0,x} \widehat\otimes_\e \cC_{0,y}$ (\cite[Chapitre I, p.~90]{zbMATH03199982}) and $\dot\cB_{xy}' = \dot\cB'_x \widehat\otimes_\e \dot\cB'_y$ which can be shown using~\cite[Theorem 3]{bargetz}.
 \item With (1) $\cH = \cD'$ or (2) $\cH = \cS'$, let $F$ be a space of distributions such that $\cH' \hookrightarrow F$ and $\cH'$ is dense in $F$. Setting $E = F'$ one easily sees that $E_{xy} \hookrightarrow \cH_x(E_y)$.
Hence, we obtain the embedding
 \[ \cD'_{L^p, xy} \hookrightarrow \cS'_x \widehat\otimes \cD'_{L^p, y} \]
 with $E = \cD'_{L^p}$, $1 \le p \le \infty$ and $F = \dot\cB$ for $p=1$ and $F = \cD_{L^q}$ with $1/p + 1/q = 1$ for $p>1$. Note however that $\cD'_{L^\infty}$ is not normal.
\end{enumerate}
\end{remark}

Considering the first part of the proof of Proposition \ref{prop_eins}, we remark that if $\cH = \cD'$ or $\cH = \cS'$, multiplication of $\delta(z-x)$ with $K(x,y)$ gives
\[
 K(x,y) \in \left\{ \begin{aligned}
             &\cD'_x(E_y) \\
             &\cS'_x(E_y) \\
            \end{aligned} \right.
\quad \Longrightarrow \quad
\left\{
\begin{aligned}
&\delta(z-x) K(x,y) \in \cD'_z \widehat \otimes \cE'_x(E_y) \\
&\delta(z-x) K(x,y) \in \cS'_z \widehat\otimes \cO'_{C,x}(E_y).
\end{aligned}
\right.
\]
Using vector valued integration with respect to $x$, these statements are in fact equivalent. As we will see from the next two propositions this equivalence can also be formulated as
\begin{equation}\label{MV}
\left\{
\begin{aligned}
 \cD'_x(E_y) &= \cO_M(\cD, \cE'_x(E_y)) \\
 \cS'_x(E_y) &= \cO_M(\cS, \cO'_{C,x}(E_y))
\end{aligned}
\right.
\end{equation}
if $\cO_M(\cK, \cL)$ is defined as follows (\cite[p.~69]{mixed} and \cite[Definition (1.1), p.~306]{zbMATH03875981}):
\begin{definition}
Let $\cK$ and $\cL$ be spaces of distributions, $\cK$ normal. Then we set
 \[ \cO_M(\cK, \cL) \coleq \{\, S \in \cD': \exists [S] \in \cL(\cK, \cL) \textrm{ such that }[S](\varphi) = \varphi \cdot S\ \forall \varphi \in \cD\,\}. \]
\end{definition}

It is well known that $\cO_M(\cS, \cS)=\cO_M(\cS',\cS')=\cO_M$, see~\cite[p.~246]{TD}, \cite[(3.12)~Ex., p.~314]{zbMATH03875981} or~\cite[Thm.~5.15, p.~86]{Petersen1983}.

Recognizing that for $S \in \cD'$ the map $\varphi \mapsto \varphi \cdot S$ in $\cL(\cD, \cD')$ has kernel $\delta(x-y)S(y) \in \cD'_{xy}$, we see:
\begin{proposition}\label{normspaces}Let $\cK$ and $\cL$ be spaces of distributions such that $\cK$ is normal and carries the topology $\gamma$. Then
\[ \cO_M(\cK, \cL) = \{ S \in \cD': \delta(x-y)S(y) \in \cK'_{c,x} \bine \cL_y \}. \]
\end{proposition}
\begin{proof}
$\delta(x-y)S(y) \in \cL_c(\cK, \cL) = \cK_c'(\cL)$ by \cite[Prop.~13, p.~52]{zbMATH03145498}.
\end{proof}
Because $\cD' = \cO_M(\cD, \cE')$ and $\cS' = \cO_M ( \cS, \cO_C')$ we can obtain the vector valued multiplier statements \eqref{MV} from the next proposition.
\begin{proposition}Let $\cK, \cL, \cM, E$ be normal spaces of distributions. Then $\cM = \cO_M(\cK, \cL)$ implies that $\cM(E) = \cO_M ( \cK, \cL(E))$.
\end{proposition}
\begin{proof}
$K \in \cM_x(E_y)$ means that the mapping
\[
 \begin{aligned}
  E'_{c,y} &\to \cM_x, \\
T(y) &\mapsto \langle K(x,y), T(y) \rangle
 \end{aligned}
\]
is linear and continuous. Because of $\cM = \cO_M(\cK, \cL)$ and Proposition \ref{normspaces} this is equivalent to continuity of
\[
\begin{aligned}
 E'_{c,y} &\to \cK'_{c,z} \bine \cL_x, \\
 T(y) &\mapsto \delta(z-x) \langle K(x,y), T(y) \rangle,
\end{aligned}
\]
i.e., $\delta(z-x) K(x,y) \in \cK'_{c,z} \bine \cL_x(E_y)$ or $K(x,y) \in \cO_M(\cK, \cL_x(E_y))$.
\end{proof}

\begin{remark}
Let us state some special cases of Proposition~\ref{prop_eins}.
\begin{enumerate}[(1)]
 \item For $\cH = \cD'$ and $E = \cD'_{L^1}$ (i.e., for partially summable kernels, cf.~\cite[p.~130]{zbMATH03145498}), Proposition \ref{prop_eins} is stated and proven in \cite[Prop.~2, pp.~539--540]{zbMATH03163214}. This proposition is a generalization of Theorem 2 in \cite[p.~24]{zbMATH03149320} from the particular kernels of the form $S(x-y)T(y)$, $S, T \in \cD'(\bR^n)$ to general kernels $K(x,y) \in \cD'(\bR^{2n}_{xy})$. In \cite[Theorem 2, p.~24]{zbMATH03149320}, the equivalences
 \begin{align*}
 S(x-y)T(y) \in \cD'_x \widehat\otimes \cD'_{L^1,y} & \Leftrightarrow \delta(z-x)S(x-y)T(y) \in \cD'_z \widehat\otimes \cD'_{L^1,xy}  \\
 & \Leftrightarrow \delta(z-x-y)S(x)T(y) \in \cD'_z \widehat\otimes \cD'_{L^1,xy} \\
 \intertext{(convolvability) and in \cite[Theorem 3, p.~26]{zbMATH03149320}, the equivalences}
 S(x-y)T(y) \in \cS'_x \widehat\otimes \cD'_{L^1,y} & \Leftrightarrow \delta(z-x)S(x-y)T(y) \in \cS'_z \widehat\otimes \cD'_{L^1, xy} \\
 & \Leftrightarrow \delta(z-x-y)S(x)T(y) \in \cS'_z \widehat\otimes \cD'_{L^1,xy}
 \end{align*}
($\cS'$-convolvability), i.e., Proposition \ref{prop_eins} for $\cH = \cD'$ or $\cS'$ and $E = \cD'_{L^1}$, are proven.

R.~Shiraishi's proof uses a parametrix of $\Delta_n$. The first proof without a parametrix is in \cite[p.~194]{zbMATH03527321}. A second proof and a generalization to $\cD'_{L^p}$ instead of $\cD'_{L^1}$ is given in \cite[Prop.~1.3.4, p.~13, and Remark 4, p.~16]{zbMATH06308371}.
\item Generalizations of the assertions in (1) to \emph{general kernels} are:
\begin{enumerate}
\item $K(x,y) \in \cD'_x \widehat\otimes \cD'_{L^p, y}$, $1 \le p \le \infty$ $\Leftrightarrow$ $\delta(z-x)K(x,y) \in \cD'_z \widehat\otimes \cD'_{L^p, xy}$ (\cite[Lemma 1, p.~321]{zbMATH05770245});
\item $K(x,y) \in \cD'_x \widehat\otimes \dot\cB'_y \Leftrightarrow \delta(z-x)K(x,y) \in \cD'_z \widehat\otimes \dot\cB'_{xy}$ (\cite[Cor.~Lemma 1, p.~324]{zbMATH05770245});
\item $K(x,y) \in \cS'_x \widehat\otimes \cD'_{L^1,y} \Leftrightarrow \delta(z-x) K(x,y) \in \cS'_z \widehat\otimes \cD'_{L^1,xy}$ (\cite[Lemma 2, p.~325]{zbMATH05770245}; \cite[p.~549, (2') $\Leftrightarrow$ (5')]{zbMATH03163214});
\item
\[
K(x,y) \in \cS'_x \widehat\otimes \left\{\begin{aligned}
&\cD'_{L^p,y} \\
&\dot\cB'_y
                                         \end{aligned} \right.
\Longleftrightarrow \delta(z-x) K(x,y) \in \cS'_z \widehat\otimes \left\{
\begin{aligned}
 &\cD'_{L^p,xy} \\
& \dot\cB'_{xy}
\end{aligned}\right.
\]
(\cite[Remark (2), Lemma 2, p.~326]{zbMATH05770245}).
\end{enumerate}
\item A characterization of \emph{semicompact kernels} is given in \cite[Proposition 30, p.~100]{zbMATH03145498}: If $K \in \cD'_{xy}$, then $K \in \cD'_x \widehat\otimes \cE'_y$ if and only if $\delta(z-x)K(x,y) \in \cD'_z \widehat\otimes \cE'_{xy}$.
\end{enumerate}

\end{remark}

\section{Regularization}\label{sec_reg}

For distributions $S \in \cD'(\bR^n)$ we find the following regularity statements in \cite{TD}:
\begin{enumerate}[(i)]
 \item $S \in \cE' \Leftrightarrow (\forall \varphi \in \cD: \varphi * S \in \cD)$ (Thm.~XI, p.~166);
 \item $S \in \cO_C' \Leftrightarrow (\forall \varphi \in \cD: \varphi * S \in \cS)$ (Thm.~IX, p. 244; Thm.~XI, p.~247);
 \item $S \in \cD'_{L^p} \Leftrightarrow (\forall \varphi \in \cD: \varphi * S \in \cD_{L^p})$ (Thm.~XXV, 2°, p. 201);
 \item[(v)] $S \in \cS' \Leftrightarrow (\forall \varphi \in \cD: \varphi * S \in \cO_C)$ (Thm~VI, 2°, p.~239; \cite[Prop.~7, p.~420]{zbMATH03230708}).
\end{enumerate}
As was shown in \cite[Lemma, p.~473]{zbMATH03951304} one can replace the spaces $\cD$, $\cS$, $\cD_{L^p}$ and $\cO_C$ on the right-hand side by the larger spaces $\cE'$, $\cO_C'$, $\cD'_{L^p}$ and $\cS'$, respectively, which means that the implications ``$\Leftarrow$'' also hold under the weaker regularity assumptions
\[ \forall \varphi \in \cD: \varphi * S \in \cE',\ \cO_C',\ \cD'_{L^p}\textrm{ and }\cS'\textrm{, respectively.} \]
By means of the closed graph theorem one can formulate these regularity statements without test functions also in terms of belonging to certain spaces:
\begin{enumerate}[(i)]
 \item $S(x-y) \in \cD'_y \widehat\otimes \cD_x$ and $S(x-y) \in \cD'_y \widehat\otimes \cE'_x$,
 \item $S(x-y) \in \cD'_y \widehat\otimes \cS_x$ and $S(x-y) \in \cD'_y \widehat\otimes \cO_{C,x}'$,
 \item $S(x-y) \in \cD'_y \widehat\otimes \cD_{L^p,x}$ and $S(x-y) \in \cD'_y \widehat\otimes \cD'_{L^p,x}$,
 \item[(v)] $S(x-y) \in \cD'_y \widehat\otimes \cO_{C,x}$ and $S(x-y) \in \cD'_y \widehat\otimes \cS'_x$.
\end{enumerate}
We furthermore have:
\begin{enumerate}
 \item[(iv)] $S \in \dot\cB' \Leftrightarrow (\forall \varphi \in \cD: \varphi * S \in \dot\cB)$
\end{enumerate}
or
\begin{enumerate}
 \item[(iv)] $S \in \dot\cB' \Leftrightarrow S(x-y) \in \cD'_y \widehat\otimes \dot\cB_x \Leftrightarrow S(x-y) \in \cD'_y \widehat\otimes \dot\cB'_x$.
\end{enumerate}
These well-known results can be summarized as follows:

\begin{proposition}\label{prop_vier}
 Let $E$ and $F$ be spaces of distributions from the same row of the following table:
\[
 \begin{aligned}
  (i)& \\
 (ii)& \\
(iii)&\\
 (iv)& \\
  (v)& \\
 \end{aligned} 
\qquad E = \left\{
 \begin{aligned}
& \cE' \\
& \cO_C' \\
& \cD'_{L^p} \quad (1 \le p < \infty) \\
& \dot\cB'  \\
& \cS' \\
 \end{aligned}
\right.
\qquad F = \left\{
 \begin{aligned}
& \cD \subseteq \cD^0 \subseteq \cE'  \\
& \cS \subseteq \cS^0 \subseteq \cO_C' \\
& \cD_{L^p} \subseteq L^p \subseteq \cD'_{L^p }\\
& \dot\cB \subseteq \cC_0 \subseteq \dot\cB' \\
& \cO_C \subseteq \cO_C^0 \subseteq \cS' \\
 \end{aligned}
\right.
\]
Then we have:
\begin{enumerate}[(1)]
 \item $ S \in E \Leftrightarrow S(x-y) \in \cD'_y \widehat\otimes F_x$;
 \item $ S \in E \Leftrightarrow S(x-y) \in \cS'_y \widehat\otimes F_x$.
\end{enumerate}
\end{proposition}

Here, $\cD^0$ is the space of continuous functions with compact support and $\cC_0 = \dot\cB^0$ the space of continuous functions rapidly vanishing at infinity;
$\cS^0$ and $\cO_C^0$ are defined in \cite[pp.~90,173]{zbMATH03230708}.

\begin{proof}``$\Rightarrow$'' follows from standard theorems of distribution theory as cited above. ``$\Leftarrow$'' can be shown using the theory of vector valued distributions by multiplication with $\delta(z-y)$. We will give the proof for (iii), the other cases being analogous. 
 
(iii): Multiplication with
\begin{align*}
 &(1)\ \delta(z-y) \in \cD'_z \widehat\otimes \cD_y && (2)\ \delta(z-y) \in \cS'_z \widehat\otimes \cS_y \\
\intertext{gives, by \cite[Prop.~25, p.~120]{zbMATH03145499},}
 &(1)\ \delta(z-y)S(x-y) \in \cD'_z \widehat\otimes \cE'_y \widehat\otimes \cD'_{L^p,x} && (2)\ \delta(z-y) S(x-y) \in \cS'_z \widehat\otimes \cO_{C,y}' \widehat\otimes \cD'_{L^p,x}. \\
\intertext{By Lemma \ref{lemma_eins} this implies}
 &(1)\ \delta(z-y)S(x-y) \in \cD'_z \widehat\otimes \cD'_{L^p,xy} && (2)\ \delta(z-y)S(x-y) \in \cS'_z \widehat\otimes \cD'_{L^p,xy}. \\
\intertext{The coordinate transformation $x-y = \xi$, $y = \eta$ and the invariance of $\cD'_{L^p,xy}$ under it then gives}
 &(1)\ \delta(z-\eta)S(\xi) \in \cD'_z \widehat\otimes \cD'_{L^p,\xi\eta} && (2)\ \delta(z-y)S(\xi) \in \cS'_z \widehat\otimes \cD'_{L^p,\xi\eta}
\end{align*}
which results in $S \in \cD'_{L^p}$.
\end{proof}

We remark that the proof of ``$\Leftarrow$'' in the case of (iii) in \cite[pp.~201--202]{TD} relies on the use of a parametrix of $\Delta_n$.

In terms of spaces of convolutors introduced by L.~Schwartz (\cite[exp.~11]{semschwartz}; \cite[p.~72]{zbMATH03145498}; \cite[(1.2) Def., p.~307]{zbMATH03875981}), Proposition \ref{prop_vier} can be formulated in a slightly different way.

\begin{definition}For spaces of distributions $\cK$ and $\cL$, $\cK$ normal, the space $\cO_C' ( \cK, \cL)$ of convolutors from $\cK$ into $\cL$ is defined as
 \[ \cO_C'(\cK, \cL) = \{\, S \in \cD'(\bR^n): \exists [S] \in \cL(\cK, \cL): [S](\varphi) = \varphi * S\ \forall \varphi \in \cD\,\}. \]
\end{definition}
Note that $\cO_{C}'(\cS,\cS) = \cO_C'(\cS',\cS') = \cO_C'$ by  \cite[(3.5)~Ex., p.~312]{zbMATH03875981} or \cite[p.~93]{Petersen1983}.

If we identify $\cK$ and $\cL$ with subspaces of $\cD'$, for $S \in \cD'$ the kernel of the convolution mapping $\varphi \mapsto \varphi * S$ is $S(x-y) \in \cD'_{xy}$. Hence, we can characterize the space of convolutors as follows:
\begin{proposition}\label{prop_zwei}Let $\cK, \cL$ be spaces of distributions such that $\cK$ is normal and carries the topology $\gamma$. Then
\[ \cO'_C(\cK, \cL) = \{\, S \in \cD': S(x-y) \in \cK'_{c,y} \bine \cL_x\,\}. \]
\end{proposition}
\begin{proof}
Again we use \cite[Prop.~13, p.~52]{zbMATH03145498} for
\[ S(x-y) \in \cL_c(\cK_y, \cL_x) \cong \cK'_{c,y} \bine \cL_x. \qedhere \]
\end{proof}
Hence we can reformulate Proposition \ref{prop_vier} under the respective conditions on the (normal) spaces of distributions $E,F$ as follows: for $S \in \cD'$, 
\begin{enumerate}[(1)]
 \item $E = \cO_C'(\cD, F)$;
 \item $E=\cO_C'(\cS,F)$.
\end{enumerate}

\begin{remark}
 \begin{enumerate}[(1)]
  \item The spaces $\cO'_C(\cK, \cK)$ (i.e., $\cL = \cK$) have been studied in \cite[p.~322]{zbMATH06256318}.
  \item A table of spaces of convolutors was also given in \cite[Theorem 5, p.~22]{zbMATH03144761}. The terminology therein differs from ours: the notations $\cK \Rightarrow \cL$, $\cL = \cK^*$ for spaces of distributions $\cK, \cL$ ($\cK$ normal) signify
\[ \cL = \{\, S \in \cD': \forall T \in \cK: S\textrm{ and }T\textrm{ are composable}\,\} \]
i.e.,
\[ \cL = \{\, S \in \cD':\forall T \in \cK: S(x-y) T(y) \in \cD'_x \widehat\otimes \cD'_{L^1,y}\,\}. \]
Denoting by $\tau$ the mapping $\tau \colon \cD' \to \cD'_x \widehat\otimes \cE_y$, $S \mapsto S(x-y)$ we have
\[ \cL = \{\, S \in \cD': \tau S \in \cO_M ( \cK, \cD' \widehat\otimes \cD'_{L^1})\, \} \]
and by Proposition \ref{normspaces} -- if $\cK$ has topology $\gamma$ --
\[ \cL = \{\, S \in \cD': \delta(z-y)S(x-y) \in \cK'_{c,z} \bine \cD'_x ( \cD'_{L^1,y})\,\}. \]
By vector valued integration with respect to $y$ and Proposition \ref{prop_zwei} this implies
\[ \cL \subseteq \cK_c' \bine \cD' = \cO_C'(\cK, \cD') \]
which proves Theorem 3 (1) in \cite{zbMATH03144761} under a weaker assumption, i.e., $\cK$ has the topology $\gamma$. In order to prove the reverse inclusion $\cL \supseteq \cO_C'(\cK, \cD')$ we assume $\cK$ to be $\dot\cB$-normal (\cite[pp.~174,178]{zbMATH03214910}): let $S \in \cD'$ with $S(x-z) \in \cK'_{c,z} \bine \cD'_x$. Multiplying with $\delta(z-y) \in \dot\cB_z \bine \cD'_{L^1,c,y}$ by \cite[Theorem 7.1, p.~31]{mixed} gives $\delta(z-y)S(x-z) \in \cK'_{c,z} \bine \cD'_{L^1,y,c}(\cD'_x)$, i.e., $S \in \cL$. Here we have to take into account the hypocontinuity of the multiplication $\dot\cB \times \cK \xrightarrow{\cdot} \cK$ 
and hence also of $\dot\cB \times \cK'_c \xrightarrow{\cdot} \cK'_c$ (\cite[Prop.~7, p.~258]{zbMATH03230708}). Using different assumptions (namely: $\cK$ a barrelled normal space of distributions with $\cK'^\vee \subseteq \cK^*$) the equality $\cK^* = \cO'_c(\cK, \cD')$ was obtained in \cite[Theorem (3.22), p.~318]{zbMATH03875981}.
 \end{enumerate}
\end{remark}

The next result extends Proposition~\ref{prop_vier} to vector valued spaces of distributions, i.e., we characterize when a kernel belongs to a completed tensor product of distribution spaces in terms of regularity properties of the kernel on the ``same level'' ($E=F$ in Prop.~\ref{prop_vier}) or ``different level'' ($E \ne F$ in Prop.~\ref{prop_vier}). This generalization of Proposition \ref{prop_vier} is motivated by the equivalence of L.~Schwartz' condition for convolvability,
\[ S(x-y) T(y) \in \cD'_x \widehat\otimes \cD'_{L^1,y} \]
for two distributions $S,T$ on $\bR^n$ (\cite[p.~130]{zbMATH03145498}), and C.~Chevalley's condition
\[ S(x-y) T(y-z) \in \cD'_{xz} \widehat\otimes L^1_y \]
(\cite[p.~112]{chevalley}; \cite[condition ($\overline{*}$), p.~19]{zbMATH03149320}; formulated without test functions in \cite[Prop.~2, p.~320]{zbMATH05770245}).

\begin{proposition}\label{prop_sechs}
 Let $E,F, \cH$ be complete normal spaces of distributions, $E$ ultrabornological, $F$ with a completing web and $\cH$ nuclear, such that the inclusions 
\[
 \xymatrix{
\cE' \ar@{^{(}->}[r] & E \ar@{^{(}->}[r] & \cH \ar@{^{(}->} [r] &
\genfrac{\lbrace}{\rbrace}{0pt}{0}{\cD'}{\cS'}\\
& F \ar@{^{(}->} [u]
}
\]
are continuous. Moreover, let $E = \cO_C' ( \genfrac{\lbrace}{\rbrace}{0pt}{0}{\cD}{\cS}, F)$ and $K \in \cD'(\bR^{2n}_{xy})$. Then:
\[ K(x,y) \in \cH_x \widehat\otimes E_y \Longleftrightarrow K(x,y-z) \in 
 \genfrac{\lbrace}{\rbrace}{0pt}{}{\cD'_z}{\cS'_z} \widehat\otimes \cH_x \widehat\otimes F_y.
\]
\end{proposition}

\begin{proof}By the closed graph theorem (see \cite[Appendix, p.~164]{MR0350361} or \cite[p.~736]{MR0306857} for its relevant extension), the mapping
 \[ \tau \colon E \to \cD'_z \widehat\otimes F_y,\quad S \mapsto S(y-z) \]
is well-defined, linear, continuous and injective. Hence, the same holds as well for
\[ \id_{\cH} \bine \tau \colon \cH_x \bine E \to \cH_x \bine (\cD'_z \widehat\otimes F_y) \]
by \cite[Prop.~4.3, p.~19 and Def.~4.3, (1--3), pp.~17,18]{zbMATH03145498}. From this we obtain
\[ K(x,y) \in \cH_x(E_y) \Longleftrightarrow (\id_{\cH_x} \bine \tau)(K) = K (x, y-z) \in \cH_x(\cD'_z(F_y)). \]
The statement with $\cS'$ in place of $\cD'$ is shown the same way.
\end{proof}
\begin{remark}
 \begin{enumerate}[(1)]
  \item For the special kernels $K(x,y) = S(x-y)T(y)$, $S,T \in \cD'(\bR^n)$ and the spaces $\cH = \cD'$, $E = \cD'_{L^p}$, $F = L^p$, $1 \le p < \infty$, Proposition \ref{prop_sechs} gives
\[ S(x-y)T(y) \in \cD'_x \widehat\otimes \cD'_{L^p,y} \Longleftrightarrow S(x-y)T(y-z) \in \cD'_{xz} \widehat\otimes L^p_y \]
(\cite[Prop.~4, p.~327]{zbMATH05770245}). For $p=1$ this is the characterization of convolvability by regularization, proven in \cite[Thm.~2, pp.~24--26]{zbMATH03149320}.
 \item Taking $K(x,y) = S(x-y)T(y)$, $S, T \in \cS'(\bR^n)$, $\cH = \cS'$, $E = \cD'_{L^1}$, $F = L^1$, Proposition \ref{prop_sechs} gives Proposition 5 of \cite{zbMATH05770245}, i.e., the characterization of $\cS'$-convolvability by regularization which was first proven in \cite[Theorem 3, p.~27]{zbMATH03149320}.
 \item Slightly generalizing Proposition \ref{prop_sechs} by taking the tensor product of the space $\cH_x \widehat\otimes E_y$ with another quasicomplete Hausdorff locally convex space $G$, we obtain with $\cH = \cD'$, $E = F = \cD'_{L^1}$ the equivalence:
\[ K(x,y) \in (\cD'_x \widehat\otimes \cD'_{L^1,y}) \bine G \Longleftrightarrow K(x, y-z) \in (\cD'_{xz} \widehat\otimes \cD'_{L^1,y} ) \bine G. \]
This is a vector-valued analogon of the characterization of partially summable distributions by regularization as in Proposition 6 of \cite[p.~545]{zbMATH03163214}. The proof therein relies on the proof of \cite[Proposition 5]{zbMATH03163214}, which in turn uses a parametrix of $\Delta_n$.
 \end{enumerate}

\end{remark}

\section{Topological properties: characterization of bounded and relatively compact sets of distributions and kernels by their regularizations}

L.~Schwartz denotes those characterizations as ``finer forms'' of regularity statements. For an example, consider \cite[pp.~201--202]{TD}:
\[ H \subseteq \cD'_{L^p},\ H\textrm{ bounded} \Longleftrightarrow (\forall \varphi \in \cD: \varphi * H \textrm{ bounded in }L^p) \]
instead of Proposition \ref{prop_vier} (iii). For the proof Schwartz employs formula (VI, 6; 23) in \cite[p.~191]{TD}, i.e.,
\[ \Delta_n^{2k} (\gamma E * \gamma E) - 2 \Delta_n^k (\gamma E * \zeta) + \zeta * \zeta = \delta \]
for a fundamental solution $E$ of $\Delta_n^{2k}$ and test functions $\gamma, \zeta \in \cD$, which amounts to a refined method of parametrices. We will prove the following ``finer'' form of Proposition \ref{prop_vier}:

\begin{proposition}\label{prop_sieben}
 Let $E$ and $F$ be normal, complete spaces of distributions such that the inclusions
\[
 \xymatrix{
\genfrac{.}{\rbrace}{0pt}{0}{\cE'}{\cO_C'} \ar@{^{(}->}[r] & E \ar@{^{(}->}[r] & \genfrac{\lbrace}{.}{0pt}{0}{\cD'}{\cS'} \\
& F \ar@{^{(}->}[u] 
},
\]
where $E=\cO_C'(\cD,F)$ or $E=\cO_C'(\cS,F)$, are continuous. Moreover, assume that $E$ is ultrabornological, $F$ has a complete web, $E_{xy}$ is invariant under linear nondegenerate coordinate transformations and $E_x \otimes_\pi E_y \hookrightarrow E_{xy} \hookrightarrow E_x \widehat\otimes \cD'_y$. Then we have the following characterization:
\begin{gather*}
 H \subseteq E \textrm{ is bounded (relatively compact)}
\\
\qquad \Longleftrightarrow H(x-y) \subseteq \cL_s ( \genfrac{\lbrace}{\rbrace}{0pt}{0}{\cD_x}{\cS_x}, F_y) \textrm{ is bounded (relatively compact)} 
\end{gather*}
i.e., $\forall \varphi \in \genfrac{\lbrace}{\rbrace}{0pt}{0}{\cD}{\cS}$ we have that $\varphi * H$ is bounded (relatively compact) in $F$ (the index ``s'' denotes the topology of pointwise convergence).
\end{proposition}
\begin{proof}
 ``$\Rightarrow$'': As in the proof of Proposition \ref{prop_sechs}, the mapping
\[ \tau: E \to \genfrac{\lbrace}{\rbrace}{0pt}{0}{\cD'_x}{\cS'_x} \widehat\otimes F_y,\ S \mapsto S(x-y) \]
is well-defined, linear and continuous. Hence, $H(x-y)$ is bounded (relatively compact) in $\cL_\e ( \genfrac{\lbrace}{\rbrace}{0pt}{0}{\cD_x}{\cS_x}, F_y)$ and therefore also in $\cL_s ( \genfrac{\lbrace}{\rbrace}{0pt}{0}{\cD_x}{\cS_x}, F_y)$.

``$\Leftarrow$'': Because $\cD$ and $\cS$ are barreled, a bounded (relatively compact) subset $H(x-y) \subseteq \cL_s ( \genfrac{\lbrace}{\rbrace}{0pt}{0}{\cD_x}{\cS_x}, F_y)$ is equicontinuous and hence also bounded (relatively compact) in $\cL_\e ( \genfrac{\lbrace}{\rbrace}{0pt}{0}{\cD_x
}{\cS_x}, F_y) = \genfrac{\lbrace}{\rbrace}{0pt}{0}{\cD'_x}{\cS'_x} \widehat\otimes F_y$. Multiplication with $\delta(z-x) \in \genfrac{\lbrace}{\rbrace}{0pt}{0}{\cD'_z \widehat\otimes \cD_x}{\cS'_z \widehat\otimes \cS_x}$ by \cite[Prop.~25, p.~120]{zbMATH03145499} gives the boundedness (relative compactness) of $\delta(z-x)H(x-y)$ in $\genfrac{\lbrace}{\rbrace}{0pt}{0}{\cD'_z \widehat\otimes \cE'_x \widehat \otimes F_y}{\cS'_z \widehat\otimes \cO'_{C,x} \widehat\otimes F_y}$ and by Lemma \ref{lemma_eins} also in $\genfrac{\lbrace}{\rbrace}{0pt}{0}{\cD'_z \widehat\otimes E_{xy}}{\cS'_z \widehat\otimes E_{xy}}$. By our assumptions on $E$, then, $\delta(z-\eta)H(\xi)$ is bounded (relatively compact) in $\cD_z' \widehat \otimes E_{\xi\eta}$, which means that $H(\xi)$ is bounded (relatively compact) in $E_\xi$.
\end{proof}

\begin{remark}In \cite{TD} one can find Proposition \ref{prop_sechs} for the spaces
\[
E = \left\{
 \begin{aligned}
& \cO_C' \\
& \cD'_{L^p} \\
& \cS' \\
& \cD'
 \end{aligned} \right.
\qquad\textrm{ and }\qquad
F = \left\{
 \begin{aligned}
& \cS \qquad\qquad\qquad &&\textrm{(p.~244)} \\
& L^p \qquad &&\textrm{(p.~202)} \\
& \cO_M \qquad &&\textrm{(p.~240)} \\
& \cD \qquad &&\textrm{(p.~194)}.
 \end{aligned}
\right.
\]
A characterization of bounded subsets of $\cD'_{L^p}$ by their regularizations can also be found in \cite[Thm.~1, p.~51]{zbMATH00569379}.
\end{remark}

For \emph{sequences of distributions} we have:
\begin{proposition}\label{prop_siebenstrich}Under the conditions of Proposition \ref{prop_sieben} on the spaces $E$ and $F$, with $(T_j)_{j \in \bN} \subseteq E$ the following equivalence holds:
 \[ T_j \to 0 \textrm{ in }E \Longleftrightarrow ( \forall \varphi \in \genfrac{\lbrace}{\rbrace}{0pt}{0}{\cD}{\cS}: \varphi * T_j \to 0 \textrm{ in }F). \]
\end{proposition}
The proof is analogous to that of Proposition \ref{prop_sieben}.

\begin{remark}
 \begin{enumerate}[(1)]
  \item For $E=F=\cD'_{L^p}$, Proposition \ref{prop_siebenstrich} gives Theorem 2 of \cite{zbMATH00569379}, p.~52. This also is remarque 2°, p.~202, in \cite{TD}.
 \item For $E = \cO_C'$ and $F = S^0$ one can find Proposition \ref{prop_siebenstrich} in \cite{TD}, p.~244 (without proof).
 \item For $E = \cD'$ and $F = E^0$, Proposition \ref{prop_siebenstrich} can be found in \cite{TD}, p.~197.
 \item For $E = F = \cS'$, Proposition \ref{prop_siebenstrich} is identical with a Lemma of Uryga (personal communication).
 \end{enumerate}
\end{remark}

An analogon of Proposition \ref{prop_sieben} for bounded (or relatively compact) sets of kernels (i.e., of vector valued distributions) and therewith a ``finer'' form of Proposition \ref{prop_sechs} is

\begin{proposition}\label{prop_acht}
 Let $E,F$ and $\cH$ be normal, complete spaces of distributions, $\cH$ nuclear, such that the inclusions
\[
 \xymatrix
{
\genfrac{.}{\rbrace}{0pt}{0}{\cE'}{\cO_C'} \ar@{^{(}->}[r] & E \ar[r] & \cH \ar@{^{(}->}[r] & \genfrac{\lbrace}{.}{0pt}{0}{\cD'}{\cS'} \\
& F \ar@{^{(}->}[u]
}
\]
are continuous. Moreover, assume $F$ has a completing web, $E$ is ultrabornological and satisfies $E_x \otimes_\pi E_y \hookrightarrow E_{xy} \hookrightarrow E_x \widehat\otimes \genfrac{\lbrace}{\rbrace}{0pt}{0}{\cD'_y}{\cS'_y}$, and $E_{xy}$ is invariant under linear non-degenerate coordinate transformations. Finally, let $E = \cO_C' ( \genfrac{\lbrace}{\rbrace}{0pt}{0}{\cD}{\cS}, F)$ and $H \subseteq \cD'(\bR^{2n}_{xy})$. Then $H(x,y) \subseteq \cH_x \widehat \otimes E_y$ is bounded (relatively compact) if and only if $H(x, y-z)$ is bounded (relatively compact) in $\cL_s ( \genfrac{\lbrace}{\rbrace}{0pt}{0}{\cD_z}{\cS_z}, \cH_x \widehat\otimes F_y)$, or equivalently: if and only if for all $\varphi \in \genfrac{\lbrace}{\rbrace}{0pt}{0}{\cD}{\cS}$ the set
\[ \varphi \underset{y}{*} H(x,y) = \{\, \langle \varphi(y-z), K(x,z) \rangle: K \in H \,\} \]
is bounded (relatively compact) in $\cH_x \widehat\otimes F_y$.
\end{proposition}
\begin{proof}
 Boundedness and relative compactness of $H(x, y-z)$ in $\genfrac{\lbrace}{\rbrace}{0pt}{0}{\cD'_z}{\cS'_z} \widehat\otimes \cH_x \widehat\otimes F_y$ is equivalent to the same property in $\cL_s ( \genfrac{\lbrace}{\rbrace}{0pt}{0}{\cD_z}{\cS_z}, \cH_x \widehat\otimes F_y)$ because $\cD_z$ and $\cS_z$ are barrelled and hence bounded sets are equicontinuous. These sets then are bounded in $\cL_{\mathfrak{S}} ( \genfrac{\lbrace}{\rbrace}{0pt}{0}{\cD_z}{\cS_z}, \cH_x \widehat\otimes F_y)$ for every $\mathfrak{S}$-topology (\cite[Thm.~2, III.~27]{zbMATH03757085}) or relatively compact in $\cL_c ( \genfrac{\lbrace}{\rbrace}{0pt}{0}{\cD_z}{\cS_z}, \cH_x \widehat\otimes F_y)$ (\cite[T.~2, XX, 3;  1, p.~310]{zbMATH04065650}).

``$\Rightarrow$'' follows from hypocontinuity of the mapping
\begin{align*}
 \id_{\cH_x} \bine \tau \colon \cH_x \widehat\otimes E_y &\to \cH_x \widehat\otimes \cD'_z \widehat\otimes F_y \\
K(x,y) &\mapsto K(x, y-z).
\end{align*}

``$\Leftarrow$'': If $H(x, y-z)$ is bounded in $\cD'_z(\cH_x \widehat\otimes F_y)$ then multiplication by $\delta(z-w) \in \cD_z \widehat\otimes \cD'_w$ by \cite[Prop.~25, p.~120]{zbMATH03145498} gives boundedness of $\delta(z-w)H(x, y-z)$ in $\cD'_w \widehat\otimes \cE'_z \widehat\otimes \cH_x \widehat\otimes F_y$ and hence boundedness of $\delta(z-w)H(x, y-z)$ in $\cD'_w \widehat\otimes \cH_x \widehat \otimes E_{yz}$. The linear change of coordinates $y-z = \xi$, $z = \eta$ then gives boundedness of $\delta(\eta - w)H(x, \xi)$ in $\cD'_w \widehat\otimes \cH_x \widehat\otimes E_{\xi\eta}\subset\cD'_w\widehat{\otimes}\cD'_\eta\widehat{\otimes}\cH_x\widehat{\otimes}E_\xi$, which by assumption is only possible if $H(x, \xi)$ is bounded in $\cH_x \widehat\otimes E_\xi$.

A similar argumentation gives the claim about relatively compact sets $H(x, y-z)$.

If $H(x, y-z)$ is bounded in $\cS'_z ( \cH_x \widehat\otimes F_y)$ (or relatively compact) then one has to multiply with $\delta(z-w) \in \cS_z \widehat\otimes \cS'_w$.
\end{proof}

\begin{remark}
 The special case $\cH = \cD'$, $E = F = \cD'_{L^1}$ in Proposition \ref{prop_acht} essentially is equivalence [1] $\Leftrightarrow$ [2] in \cite[Prop.~5, p.543]{zbMATH03163214} (characterization of a relatively compact set of partially summable kernels).
\end{remark}

Finally, we remark that instead of characterizations of bounded (or relatively compact) sets $H(x,y)$ in $\cH_x \widehat\otimes E_y$ by the boundedness (or relative compactness) of their regularizations $H(x, y-z)$ in $\genfrac{\lbrace}{\rbrace}{0pt}{0}{\cD'_z}{\cS'_z} \widehat\otimes \cH_x \widehat\otimes F_y$ or in $\cL_s ( \genfrac{\lbrace}{\rbrace}{0pt}{0}{\cD_z}{\cS_z}, \cH_x \widehat\otimes F_y)$, characterizations can also be given in terms of properties of $\delta(z-y) H(x,y)$ or of $\delta(z-y)H(x, y-z)$. An example for this is the characterization of relatively compact sets $H$ of partially summable kernels (i.e., $H(x,y) \subseteq \cD'_x \widehat\otimes \cD'_{L^1, y}$) in \cite[Prop.~5, p.~243]{zbMATH03163214} ((1) $\Leftrightarrow$ (3), (1) $\Leftrightarrow$ (4)).

\section*{Acknowledgements}

In the preparation of this work, E.~A. Nigsch was supported by grant P26859-N25 of the Austrian Science Fund (FWF).




\end{document}